\documentclass[11pt]{amsart}
\usepackage[foot]{amsaddr}
\usepackage[english]{babel}
\usepackage[
pdftex,hyperfootnotes]
{hyperref}
\usepackage[usenames,dvipsnames,svgnames,table,x11names]{xcolor}
\usepackage{pgfplots}
\usepackage{tikz}
\usepackage{tikz-3dplot}
\usetikzlibrary{calc,shadows,shapes.callouts,shapes.geometric,shapes.misc,positioning,patterns,%
	decorations.pathmorphing,decorations.markings,decorations.fractals,decorations.pathreplacing,shadings,fadings}
\pgfplotsset{compat=1.14} 
\usepackage[utf8]{inputenc}
\usepackage{nicematrix,booktabs}

\usepackage{comment}
\usepackage[textwidth=17.5cm,textheight=22.275cm,
height=
24.275cm,
width=18cm]{geometry}
\usepackage{rotating,multirow,pdflscape}
\usepackage{amssymb,latexsym,amsmath,amsthm}
\usepackage{mathrsfs}
\usepackage{mathtools}
\usepackage{bigints}
\usepackage{tensor}

\mathtoolsset{centercolon}
\usepackage[toc,page]{appendix}

\usepackage{tocvsec2}

\usepackage{Baskervaldx}
\newtheorem{definition}{{ Definition}}
\newtheorem{theorem}{Theorem}
\newtheorem{pro}{ Proposition }
\newtheorem{lemma}{Lemma}

\newtheorem{rem}{Remark}

\newcommand{\ii}{\operatorname{i}}

\renewcommand{\d}{\operatorname{d}}
\newcommand{\Exp}[1]{\operatorname{e}^{#1}}

\newcommand{\diag}{\operatorname{diag}}

\newcommand{\C}{\mathbb{C}}
\newcommand{\T}{\mathbb{T}}

\newcommand{\N}{\mathbb{N}}
\newcommand{\R}{\mathbb{R}}
\newcommand{\Z}{\mathbb{Z}}

\newcommand{\res}[2]{\operatorname{Res}\left(#1,#2\right)}

\makeatletter
\def\@settitle{\begin{center}%
		\baselineskip14\p@\relax
		\bfseries
		\uppercasenonmath\@title
		\@title
		\ifx\@subtitle\@empty\else
		\\[1ex]\uppercasenonmath\@subtitle
		\footnotesize\mdseries\@subtitle
		\fi
	\end{center}%
}
\def\subtitle#1{\gdef\@subtitle{#1}}
\def\@subtitle{}
\makeatother

\begin{document}
\title[Discrete Hypergeometrical Orthogonal Polynomials]{Laguerre--Freud Equations for  three families of 
	 hypergeometrical discrete orthogonal polynomials}

%
\author[I Fernández-Irisarri]{Itsaso Fernández-Irisarri$^1$}
\email{$^1$itsasofe@ucm.es}
\address{$^1$Departamento de Física Teórica, Universidad Complutense de Madrid, Plaza Ciencias 1, 28040-Madrid, Spain}

\author[M Mañas]{Manuel Mañas$^2$}
\email{$^2$manuel.manas@ucm.es}

\address{$^2$Departamento de Física Teórica, Universidad Complutense de Madrid, Plaza Ciencias 1, 28040-Madrid, Spain \&
	Instituto de Ciencias Matematicas (ICMAT), Campus de Cantoblanco UAM, 28049-Madrid, Spain}

\thanks{The authors thanks financial support from the Spanish ``Agencia Estatal de Investigación" research project [PGC2018-096504-B-C33], \emph{Ortogonalidad y Aproximación: Teoría y Aplicaciones en Física Matemática} and research project [PID2021- 122154NB-I00], \emph{Ortogonalidad y aproximación con aplicaciones en machine learning y teoría de la probabilidad}.}

\begin{abstract}
The Cholesky factorization of the moment matrix is considered for 
discrete orthogonal polynomials of hypergeometrical type. We derive the Laguerre--Freud equations when the  first moments of the weights are given by 
the  ${}_1F_2$, ${}_2F_2$ and ${}_3F_2$ generalized hypergeometrical functions.

\end{abstract}


\maketitle

\allowdisplaybreaks
\section{Introduction}

Discrete orthogonal polynomials is  an important and active area of the theory of orthogonal polynomials see \cite{NSU,baik}  as well as \cite{Ismail,Ismail2,Beals_Wong,walter}. When a discrete Pearson equation is fulfilled by the weight we are dealing with semiclassical discrete orthogonal polynomials,  see \cite{diego_paco,diego_paco1  ,diego,diego1}. For some specific type of weights of generalized Charlier and Meixner types, the corresponding  Freud--Laguerre type equations for the coefficients of the three term recurrence has been studied, see for example \cite{clarkson,filipuk_vanassche0,filipuk_vanassche1,filipuk_vanassche2,smet_vanassche}.

This paper is the final one in  our series of  works on the of the Gauss--Borel factorization of the corresponding moment matrix, see \cite{intro}, in   standard discrete orthogonality. In \cite{Manas_Fernandez-Irrisarri}  the Cholesky factorization of the moment matrix was used to  study  discrete orthogonal polynomials  on the homogeneous lattice.  For weights subject to a discrete Pearson equation, so that the moments  where logarithmic derivatives of generalized hypergeometric functions,  a banded semi-infinite matrix  $\Psi$, the so called Laguerre--Freud structure matrix,  that models the shifts by $\pm 1$  in the independent variable of the sequence  of orthogonal polynomials was given.  
Contiguous   relations for the generalized  hypergeometric functions translate  into symmetries for the corresponding moment matrix, and  the 3D Nijhoff--Capel  discrete Toda lattice \cite{nijhoff,Hietarinta}  describes the corresponding contiguous shifts for the squared norms of the orthogonal polynomials. In \cite{Manas} we gave an interpretation for the  contiguous transformations for the generalized hypergeometric functions in terms of simple Christoffel and Geronimus transformations. Using  the Geronimus--Uvarov perturbations we got determinantal expressions for the shifted orthogonal polynomials.
Then, in \cite{Irrisari-Manas} we discussed  the generalized Charlier, Meixner and type I Hahn discrete orthogonal polynomials (according to the nomenclature in \cite{diego_paco}),   analyzed the  Laguerre--Freud structure matrix $\Psi$, and derive from its banded structure and its compatibility with the Toda equation and the Jacobi matrix, a number of nonlinear equations for the coefficients $\{\beta_{n},\gamma_n\}$ of the three term recursion relations $zP_n(z)=P_{n+1}(z)+\beta_n P_n(z)+\gamma_n P_{n-1}(z)$ satisfied by the orthogonal polynomial sequence.
These Laguerre--Freud equations are of the form
$	\gamma_{n+1} =\mathscr G (n,\gamma_n,\gamma_{n-1},\dots,\beta_n,\beta_{n-1}\dots)$ and $\beta_{n+1 }= \mathscr B (n,\gamma_{n+1},\gamma_n,\dots,\beta_n,\beta_{n-1},\dots)$,
for some functions $\mathscr B,\mathscr G$. Magnus  \cite{magnus,magnus1,magnus2,magnus3} named these type of relations, attending to \cite{laguerre,freud}, as Laguerre--Freud relations.
There are several papers discussing Laguerre--Freud relations for the generalized Charlier, generalized Meixner and type I generalized Hahn cases, see \cite{smet_vanassche,clarkson,filipuk_vanassche0,filipuk_vanassche1,filipuk_vanassche2,diego}. 

In this paper we conclude our studies on these themes, and extend these methods to three  families of discrete orthogonal polynomials with first moments given by the generalized hypergeometric functions  ${}_1F_2$, ${}_2F_2$ and ${}_3F_2$. Notice that generalized Charlier, Meixner and type Hahn correspond to  ${}_0F_1$,  ${}_1F_1$ and  ${}_2F_1$, respectively.
For the three families  we explicitly compute the Laguerre-Freud structure matrix which happens to be a pentadiagonal, hexadiagonal or heptadiagonal banded matrices, respectively. We also compute the Laguerre--Freud  equations and give explicit expressions for the leading nontrivial coefficients of the orthogonal polynomials in terms of the recursion coefficients.

The layout of the paper is as follows. We first complete this introduction by discussing, without proof but referring to the appropriate sources, some preliminary material. Then, we have three sections devoted each of them to the corresponding hypergeometrical function, ${}_1F_2$, ${}_2F_2$ and ${}_3F_2$, finding in each case the Laguerre--Freud structure matrix and corresponding Laguerre--Freud equations. 

\subsection{Pearson equations and discrete orthogonal polynomials} 

Let us recall some relevant facts of the theory of orthogonal polynomials. 
Given a linear functional $\rho_z\in\C^*[z]$,  the corresponding   moment  matrix is
 \begin{align*}
 G&=(G_{n,m}), &G_{n,m}&=\rho_{n+m}, &\rho_n&= \big\langle\rho_z,z^n\big\rangle, & n,m\in\N_0:=\{0,1,2,\dots\},
 \end{align*}
with $\rho_n$ being the $n$-th moment of the linear functional $\rho_z$.
Let us assume that the moment matrix is such that  all its truncations, which are Hankel matrices, $G_{i+1,j}=G_{i,j+1}$,
 \begin{align*}
 G^{[k]}=\begin{bNiceMatrix}[small]
 G_{0,0}&\Cdots &G_{0,k-1}\\
 \Vdots & & \Vdots\\
 G_{k-1,0}&\Cdots & G_{k-1,k-1}
 \end{bNiceMatrix}=\begin{bNiceMatrix}[
 small]
 \rho_{0}&\rho_1&\rho_2&\Cdots &\rho_{k-1}\\
 \rho_1 &\rho_2 &&\Iddots& \rho_k\\
 \rho_2&&&&\Vdots\\
 \Vdots& &\Iddots&&\\[9pt]
 \rho_{k-1}&\rho_k&\Cdots &&\rho_{2k-2}
 \end{bNiceMatrix}
 \end{align*}
 are nonsingular; i.e. the Hankel determinants $\varDelta_k:=\det G^{[k]} $ do not cancel, $\varDelta_k\neq 0 $, $k\in\N_0$. Then, there are  monic polynomials
 \begin{align}\label{eq:polynomials}
 P_n(z)&=z^n+p^1_n z^{n-1}+\dots+p_n^n, & n&\in\N_0,
 \end{align}
 satisfying the following orthogonality relations 
$  \big\langle \rho, P_n(z)z^k\big\rangle =0$,  for $ k\in\{0,\dots,n-1\}$ and 
   $\big\langle \rho, P_n(z)z^n\big\rangle =H_n\neq 0$,
 and $\{P_n(z)\}_{n\in\N_0}$ is a sequence of  orthogonal polynomials, i.e., 
$ \big\langle\rho,P_n(z)P_m(z)\big\rangle=\delta_{n,m}H_n$ for $n,m\in\N_0$.
The   symmetric bilinear form $ \langle F, G\rangle_\rho:=\langle \rho, FG\rangle$,
is  such that the moment matrix is the Gram matrix of this bilinear form and
$\langle P_n, P_m\rangle_\rho:=\delta_{n,m} H_n$.

Introducing  he monomial sequence
$\chi(z):=\left(\begin{NiceMatrix}
1&z&z^2&\Cdots
\end{NiceMatrix}\right)^\top$
the moment  matrix is  $G=\left\langle\rho, \chi\chi^\top\right\rangle$,
and $\chi$ is an eigenvector of the \emph{shift matrix}, $\Lambda \chi=x\chi$, where
\begin{align*}
\Lambda:=\left[\begin{NiceMatrix}[columns-width = auto,small]
0 & 1 & 0 &\Cdots\\
0&0&1&\Ddots\\
\Vdots& \Ddots&\Ddots &\Ddots
\end{NiceMatrix}\right].
\end{align*}
Hence
$\Lambda G=G\Lambda^\top$, this equation is equivalent to the fact that  the moment matrix is a Hankel.
As the moment matrix is symmetric its Borel--Gauss factorization is a Cholesky factorization,  i.e.
\begin{align}\label{eq:Cholesky}
G=S^{-1}HS^{-\top},
\end{align}
where $S$ is a lower unitriangular matrix that can be written as 
\begin{align*}
	S=\left[\begin{NiceMatrix}[columns-width = auto,small]
	1 & 0 &\Cdots &\\
	S_{1,0 } &  1&\Ddots&\\
	S_{2,0} & S_{2,1} & \Ddots &\\
	\Vdots & \Ddots& \Ddots& 
	\end{NiceMatrix}\right],
	\end{align*}
	and $H=\diag(H_0,H_1,\dots)$ is a  diagonal matrix, with $H_k\neq 0$, for $k\in\N_0$.
	The Cholesky  factorization does hold whenever the principal minors of the moment matrix; i.e., the Hankel determinants $\varDelta_k$,  do not cancel.
	
	The entries $P_n(z)$ of the polynomial sequence 
	\begin{align}\label{eq:PS}
	P(z):=S\chi(z),
	\end{align}
	are the monic orthogonal polynomials of the functional $\rho$. 

	We have the determinantal expressions $	H_{k}=\frac{\varDelta_{k+1}}{\varDelta_k}$ and $p^1_k=-\frac{\tilde \varDelta_k}{\varDelta_k}$,
	with 
	\begin{align*}
	 \varDelta_k&:=\begin{vNiceMatrix}[small]
\rho_{0}&\Cdots & &\rho_{k-2}&\rho_{k-1}\\
\Vdots    &                 &\Iddots& \Iddots&\Vdots\\
   &                 &                &                  &\\
\rho_{k-2}&             &                &&\rho_{2k-3}\\[14pt]
\rho_{k-1}& \Cdots&              &\rho_{2k-3}&\rho_{2k-2}
\end{vNiceMatrix}, 
&	\tilde \varDelta_k&:=\begin{vNiceMatrix}[small]
	\rho_{0}&\Cdots& &\rho_{k-2}&\rho_k\\
	\Vdots & &\Iddots& \rho_{k-1}&\Vdots\\
	& &\Iddots&\\
		\rho_{k-2}& & &&\rho_{2k-2}\\[10pt]
	\rho_{k-1}& \Cdots& &\rho_{2k-3}&\rho_{2k-1}
	\end{vNiceMatrix}.
	\end{align*}
The lower Hessenberg semi-infinite matrix
\begin{align}\label{eq:Jacobi}
J=S\Lambda S^{-1}
\end{align}
 has the polynomial sequence $P(z)$ as eigenvector with eigenvalue $z$. 
The Hankel condition  $\Lambda G=G\Lambda^\top$ and the Cholesky factorization gives
\begin{align}\label{eq:symmetry_J}
J H=(JH)^\top =HJ^\top.
\end{align}
As  the Hessenberg matrix $JH$ is symmetric the Jacobi matrix $J$  is tridiagonal, i.e.
\begin{align*}
J=\left[\begin{NiceMatrix}[
	small]
\beta_0 & 1& 0&\Cdots& \\
\gamma_1 &\beta_1 & 1 &\Ddots&\\[8pt]
0 &\gamma_2 &\beta_2 &1 &\\
\Vdots&\Ddots& \Ddots& \Ddots&\Ddots 
\end{NiceMatrix}\right]
\end{align*}
and we find  the three term recursion relation 
$zP_n(z)=P_{n+1}(z)+\beta_n P_n(z)+\gamma_n P_{n-1}(z)$,
that with the  initial conditions $P_{-1}=0$ and $P_0=1$ completely determines   the sequence of orthogonal polynomials $\{P_n(z)\}_{n\in\N_0}$ in terms of the recursion coefficients $\beta_n,\gamma_n$.
	The recursion coefficients are given by
	\begin{align}\label{eq:equations0}
	\beta_n&=p_n^1-p_{n+1}^1=-\frac{\tilde \varDelta_n}{\varDelta_n}+\frac{\tilde \varDelta_{n+1}}{\varDelta_{n+1}},&  \gamma_{n+1}&=\frac{H_{n+1}}{H_{n}}=\frac{\varDelta_{n+1}\varDelta_{n-1}}{\varDelta_n^2},& n\in\N_0,
	\end{align}

We introduce the following diagonal matrices
$ \gamma:=\diag (\gamma_1,\gamma_2,\dots)$ and $\beta:=\diag(\beta_0 ,\beta_{1},\dots)$.
In general, given any semi-infinite matrix $A$, we will write $A=A_-+A_+$, where $A_-$ is a strictly lower triangular matrix and $A_+$ an upper triangular matrix. Moreover, $A_0$ will denote the diagonal part of  $A$.

The lower Pascal matrix  is 
\begin{align}\label{eq:Pascal_matrix}
B&=(B_{n,m}), & B_{n,m}&:= \begin{cases}
\displaystyle \binom{n}{m}, & n\geq m,\\
0, &n<m.
\end{cases}
\end{align}
so that
\begin{align}\label{eq:Pascal}
\chi(z+1)=B\chi(z).	
\end{align}
The inverse of this Pascal matrix is
\begin{align*}
B^{-1}&=[\tilde B_{n,m}], & \tilde B_{n,m}&:= \begin{cases}
(-1)^{n+m}\displaystyle \binom{n}{m}, & n\geq m,\\
0, &n<m.
\end{cases}
\end{align*}
and
$\chi(z-1)=B^{-1}\chi(z)$.	
 The    \emph{dressed Pascal matrices} are
$\Pi:=SBS^{-1}$ and  $ \Pi^{-1}:=SB^{-1}S^{-1}$ and give he connection between shifted and non shifted polynomials, i.e.
\begin{align}\label{eq:PascalP}
P(z+1)&=\Pi P(z), & P(z-1)&=\Pi^{-1}P(z).
\end{align}
The lower Pascal matrix can be expressed  as follows
\begin{align*}
B^{\pm 1}&=I\pm\Lambda^\top D+\big(\Lambda^\top\big)^2D^{[2]}\pm\big(\Lambda^\top\big)^3D^{[3]}+\cdots,
\end{align*}
 where 
$ D=\diag(1,2,3,\dots)$ and  $D^{[k]}:=\frac{1}{k}\diag\big(k^{(k)}, (k+1)^{(k)},(k+2)^{(k)}\cdots\big)$, 
 in terms of the falling factorials
$ x^{(k)}:=x(x-1)(x-2)\cdots (x-k+1)$.
That is,  
\begin{align*}
D^{[k]}_n&=\frac{(n+k)\cdots (n+1)}{k}, & k&\in\N, & n&\in\N_0.
\end{align*}
The lower unitriangular factor can be  written as
$S=I+\Lambda^\top S^{[1]}+\big(\Lambda^\top\big)^2S^{[2]}+\cdots$
with 
$S^{[k]}=\diag \big(S^{[k]}_0, S^{[k]}_1,\dots\big)$.  
We will use the \emph{shift operators} $T_\pm$ acting over the diagonal matrices as follows
\begin{align*}
T_-\diag(a_0,a_1,\dots)&:=\diag (a_1, a_2,\dots),&
T_+\diag(a_0,a_1,\dots)&:=\diag(0,a_0,a_1,\dots).
\end{align*}
For any diagonal matrix $A=\diag(A_0,A_1,\dots)$ we have
\begin{align}
	\Lambda A&=(T_-A)\Lambda,&   A	\Lambda &=\Lambda (T_+A), 	& A \Lambda^\top  &=\Lambda^\top(T_-A), &\Lambda^\top  A &= (T_+A)\Lambda ^\top.
\end{align}
In terms of these shift operators we find
\begin{align*}
	2D^{[2]}&=(T_-D)D, & 3D^{[3]}&=(T_-^2D)(T_-D)D=2(T_-D^{[2]})D=2D^{[2]}(T_-^2 D).
\end{align*}
For the inverse matrix 
\begin{align*}
S^{-1}&=I+\Lambda^\top S^{[-1]}+\big(\Lambda^\top\big)^2 S^{[-2]}+\cdots,
\end{align*}
we have
\begin{align*}
 S^{[-1]}&=-S^{[1]}, \\S^{[-2]}&=-S^{[2]}+(T_-S^{[1]})S^{[1]},\\
S^{[-3]}&=-S^{[3]}+(T_-S^{[2]})S^{[1]}
+(T_-^2S^{[1]})S^{[2]}
-(T_-^2 S^{[1]})
(T_-S^{[1]})S^{[1]}.
\end{align*}
	Corresponding expansions  for the dressed Pascal matrices  are
\begin{align*}
\Pi^{\pm 1}=I+\Lambda^\top \pi^{[\pm 1]}+(\Lambda^\top)^2  \pi^{[\pm 2]}+\cdots
\end{align*}
with $\pi^{[\pm n]}=\diag(\pi^{[\pm n]}_0,\pi^{[\pm n]}_1,\dots)$. 
	We have
	\begin{gather}\label{eq:pis}
	\begin{aligned}
	\pi^{[\pm 1]}_n&=\pm (n+1), &
	\pi^{[\pm 2]}_n&=\frac{(n+2)(n+1)}{2}\pm
	p^1_{n+2}(n+1)\mp (n+2) p^{1}_{n+1}\\&&&=\frac{(n+2)(n+1)}{2}\mp (n+1)\beta_{n+1}
	\mp  p^{1}_{n+1},
	\end{aligned}\\\label{eq:piss}
	\begin{multlined}[t][.9\textwidth]
	\pi^{[\pm 3]}_n=\pm\frac{(n+3)(n+2)(n+1)}{3}+\frac{(n+2)(n+1)}{2}p^1_{n+3}-
	\frac{(n+3)(n+2)}{2}p^1_{n+1}\\\pm (n+1) p^2_{n+3}\mp (n+3)p^2_{n+2}\pm
	(n+3)p^1_{n+2}p^1_{n+1}\mp(n+2)p^1_{n+3}p^1_{n+1}.\end{multlined}
	\end{gather}
Moreover, the following relations are fulfilled
	\begin{gather}\label{eq:pis2}
	\begin{aligned}
		\pi^{[1]}+\pi^{[-1]}&=0, &\pi^{[2]}+\pi^{[-2]}&=2D^{[2]}, &\pi^{[3]}+\pi^{[-3]}&=2((T_-^2S^{[1]})D^{[2]}-(T_-D^{[2]})S^{[1]}).
	\end{aligned}
	\end{gather}

We now recall some facts regarding discrete orthogonality. Hence, we focus on measures with  support  on the homogeneous lattice  $\N_0$, i.e.
$\rho=\sum_{k=0}^\infty \delta(z-k) w(k)$,
with  moments given by
\begin{align}\label{eq:moments}
\rho_n=\sum_{k=0}^\infty k^n w(k),
\end{align}
and, in particular,  with $0$-th moment given by
\begin{align}\label{eq:first_moment}
\rho_0=\sum_{k=0}^\infty w(k).
\end{align}

The weights  we consider in this paper satisfy the following  \emph{discrete Pearson equation}
\begin{align}\label{eq:Pearson0}
\nabla (\sigma w)&=\tau w,
\end{align}
that is $\sigma(k) w(k)-\sigma(k-1) w(k-1)=\tau(k)w(k)$, for $k\in\N$, 
with $\sigma(z),\tau(z)\in\R[z]$.
If we write  $\theta:=\tau-\sigma$, the previous Pearson equation reads
\begin{align}\label{eq:Pearson}
\theta(k+1)w(k+1)&=\sigma(k)w(k), &
k\in\N_0.
\end{align}

If  $N+1:=\deg\theta(z)$ and  $M:=\deg\sigma(z)$,
and the  zeros of these polynomials are  $\{-b_i+1\}_{i=1}^{N}$ and $\{-a_i\}_{i=1}^M$
we write
$\theta(z)= z(z+b_1-1)\cdots(z+b_{N}-1)$ and $\sigma(z)= \eta (z+a_1)\cdots(z+a_M)$.
According to \eqref{eq:first_moment} the $0$-th moment  
\begin{align*}
\rho_0&=\sum_{k=0}^\infty w(k)=\sum_{k=0}^\infty \frac{(a_1)_k\cdots(a_M)_k}{(b_1+1)_k\cdots(b_{N}+1)_k}\frac{\eta^k}{k!}=\tensor[_M]{F}{_{N}} (a_1,\dots,a_M;b_1,\dots,b_{N};\eta)
=
{\displaystyle \,{}_{M}F_{N}\left[{\begin{matrix}a_{1}&\cdots &a_{M}\\b_{1}&\cdots &b_{N}\end{matrix}};\eta\right].}
\end{align*}
is the generalized hypergeometric function, where we are using the two standard 
notations,
see \cite{generalized_hypegeometric_functions,slater}.
Then,  according to
\eqref{eq:moments}, for $n\in\N$, the corresponding  higher moments  $\rho_n=\sum_{k=0}^\infty k^n w(k)$, are
\begin{align*}
\rho_n&=\vartheta_\eta^n\rho_0=\vartheta_\eta^n\Big({\displaystyle \,{}_{M}F_{N}\left[{\begin{matrix}a_{1}&\cdots &a_{M}\\b_{1}&\cdots &b_{N}\end{matrix}};\eta\right]}\Big), &\vartheta_\eta:=\eta\frac{\partial }{\partial \eta}.
\end{align*}

Given a function $f(\eta)$, we consider the Wronskian
\begin{align*}
\mathscr W_k(f):=\begin{vNiceMatrix}[
	small]
f &\vartheta_\eta f& \vartheta_\eta^2f&\Cdots &\vartheta_\eta^kf\\
	\vartheta_\eta f& \vartheta_\eta^2f&&\Iddots& \vartheta_\eta^{k+1}f\\
	\vartheta_\eta^2 f&&&&\Vdots\\
	\Vdots& &\Iddots&&\\
	\vartheta_z^kf&\vartheta_\eta^{k+1}f&\Cdots&&\vartheta_\eta^{2k}f
\end{vNiceMatrix}.
\end{align*}
Let us  introduce the following tau-functions
\begin{align*}
\tau_k&:=\mathscr W_{k}\Big({\displaystyle \,{}_{M}F_{N}\left[{\begin{matrix}a_{1}&\cdots &a_{M}\\b_{1}&\cdots &b_{N}\end{matrix}};\eta\right]}\Big),
\end{align*}
i.e. Wronskians of generalized hypergeometric functions.
Then, we have that 
\begin{align*}
\varDelta_k&=\tau_k, & 
\tilde \varDelta_k &=\vartheta_\eta\tau_k, &H_k&=\frac{\tau_{k+1}}{\tau_k},	&
p^1_k&=-\vartheta_\eta\log \tau_k.
\end{align*}

In \cite{Manas_Fernandez-Irrisarri} it was shown that:
\begin{theorem}[Hypergeometric symmetries]
	Let the weight $w$ be subject to a discrete Pearson equation of the type \eqref{eq:Pearson}, where the functions $\theta,\sigma$  are  polynomials, with  $\theta(0)=0$. Then, 
	\begin{enumerate}
		\item 	The  moment matrix fulfills
		\begin{align}\label{eq:Gram symmetry}
			\theta(\Lambda)G=B\sigma(\Lambda)GB^\top.
		\end{align}
		\item The  Jacobi matrix satisfies 
		\begin{align}\label{eq:Jacobi symmetry}
			\Pi^{-1}	H\theta(J^\top)=\sigma(J)H\Pi^\top,
		\end{align}	
		and the matrices $H\theta(J^\top)$ and $\sigma(J)H$ are symmetric.
	\end{enumerate}
\end{theorem}
\begin{theorem}[Laguerre--Freud structure matrix]
	Let us assume that the weight $w$  solves the discrete Pearson equation \eqref{eq:Pearson} with  $\theta,\sigma$ polynomials such that $\theta(0)=0$, $\deg\theta(z)=N+1$, $ \deg\sigma(z)=M$. 	Then,:
	\begin{enumerate}
	\item The Laguerre--Freud structure matrix
\begin{align}\label{eq:Psi}
\Psi&:=\Pi^{-1}H\theta(J^\top)=\sigma(J)H\Pi^\top=\Pi^{-1}\theta(J)H=H\sigma(J^\top)\Pi^\top\\
&=\theta(J+I)\Pi^{-1} H=H\Pi^\top\sigma(J^\top-I),\label{eq:Psi2}
\end{align}
has only  $N+M+2$ possibly nonzero  diagonals ($N+1$ superdiagonals and $M$ subdiagonals) 
\begin{align*}
\Psi=(\Lambda^\top)^M\psi^{(-M)}+\dots+\Lambda^\top \psi^{(-1)}+\psi^{(0)}+
\psi^{(1)}\Lambda+\dots+\psi^{(N+1)}\Lambda^{N+1},
\end{align*}
for some diagonal matrices $\psi^{(k)}$. In particular,  the lowest subdiagonal and highest superdiagonal  are given by
\begin{align}\label{eq:diagonals_Psi}
\left\{
\begin{aligned}
(\Lambda^\top)^M\psi^{(-M)}&=\eta(J_-)^MH,&
\psi^{(-M)}=\eta H\prod_{k=0}^{M-1}T_-^k\gamma=\eta\diag\Big(H_0\prod_{k=1}^{M}\gamma_k, H_1\prod_{k=2}^{M+1}\gamma_k,\dots\Big),\\
\psi^{(N+1)} \Lambda^{N+1}&=H(J_-^\top)^{N+1},&
\psi^{(N+1)}=H\prod_{k=0}^{N}T_-^k\gamma=\diag\Big(H_0\prod_{k=1}^{N+1}\gamma_k, H_1\prod_{k=2}^{N+2}\gamma_k,\dots\Big).
\end{aligned}
\right.
\end{align}
\item The vector $P(z)$ of orthogonal polynomials fulfill the following structure equations
\begin{align}\label{eq:P_shift}
\theta(z)P(z-1)&=\Psi H^{-1} P(z), &
\sigma(z)P(z+1)&=\Psi^\top H^{-1} P(z).
\end{align}
\item 	The following compatibility conditions for the Laguerre--Freud  and Jacobi matrices  hold
\begin{subequations}
	\begin{align}\label{eq:compatibility_Jacobi_structure_a}
	[\Psi H^{-1},J]&=\Psi H^{-1}, \\ \label{eq:compatibility_Jacobi_structure_b}	[J, \Psi ^\top H^{-1}]&=\Psi ^\top H^{-1}.
	\end{align}
\end{subequations}
\end{enumerate}
\end{theorem}

We now recall some facts regarding the Toda flows appearing in this discrete orthogonal polynomial setting.
Let us define the strictly lower triangular matrix
$\Phi:=(\vartheta_{\eta} S ) S^{-1}$.
Then, see \cite{Manas_Fernandez-Irrisarri}, we have
that the orthogonal polynomial sequence  $P$ fulfills
	$\vartheta_{\eta} P=\Phi  P$,
and the Sato--Wilson equations 
$-\Phi  H+\vartheta_\eta H-H \Phi^\top=JH$
	are satisfied.
	Consequently, $\Phi=-J_-$ and, for $n\in\N_0$, we have 
$\vartheta_\eta \log H_n=J_{n,n}$.
 Moreover,
$	\Phi =(\vartheta_{\eta} S)S^{-1}=-\Lambda^\top\gamma$ and $	(\vartheta_\eta H) H^{-1}=\beta$ are satisfied.
	The functions $q_n:=\log H_n$, $n\in\N$, satisfy the Toda equations
	\begin{align}\label{eq:Toda_equation}
	\vartheta_\eta^2q_n=\Exp{q_{n+1}-q_n}-\Exp{q_n-q_{n-1}}.
	\end{align}
	For $n\in\N$, we also have
$\vartheta_\eta P_{n}(z)=-\gamma_n P_{n-1}(z)$. Additionally, the  Lax equation 
	$	\vartheta_\eta J=[J_+,J]$ is fulfilled,
	and  recursion coefficients satisfy the following Toda system, $	\vartheta_\eta\beta_n=\gamma_{n+1}-\gamma_n,
	$ and  $\vartheta_\eta\log\gamma_n=\beta_{n}-\beta_{n-1}$,
for $n\in\N_0$ and $\beta_{-1}=0$. Thus, we get 
$	\vartheta_\eta^2\log\gamma_n+2\gamma_n=\gamma_{n+1}+\gamma_{n-1}$.

From the compatibility of 
$P(z+1)=\Pi P(z)$ and $\vartheta_\eta (P(z))=\Phi  P(z)$ 
 we get
$\vartheta_\eta\Pi =[\Phi ,\Pi ]$.
In the general case, the dressed Pascal matrix $\Pi$ is a lower unitriangular semi-infinite matrix,  that possibly has an infinite number of subdiagonals. However, when the weight $w(z)=v(z)\eta^z$ satisfies the Pearson equation \eqref{eq:Pearson}, with $v$ independent of $\eta$, that is
$\theta(k+1)v(k+1)\eta=\sigma(k)v(k)$,
the situation improves as we have the banded semi-infinite matrix $\Psi$ that models the shift in the $z$ variable as in \eqref{eq:P_shift}. From the previous discrete 
Pearson equation we see that  $\sigma(z)=\eta\kappa(z)$ with $\kappa,\theta$  $\eta$-independent polynomials in $z$
\begin{align}\label{eq:Pearson_Toda}
\theta(k+1)v(k+1)=\eta\kappa(k)v(k).
\end{align}

\begin{pro}[\cite{Manas_Fernandez-Irrisarri}]
	Let us assume   a weight $w$ satisfying the Pearson equation
	\eqref{eq:Pearson}. Then, the Laguerre--Freud structure matrix $ \Psi$ given in \eqref{eq:Psi} satisfies
\begin{subequations}
		\begin{align}\label{eq:eta_compatibility_Pearson_1a}
	\vartheta_\eta(\eta^{-1}\Psi^\top H^{-1} )&=[\Phi ,\eta^{-1}\Psi^\top H^{-1}  ], \\
	\vartheta_\eta(\Psi H^{-1} )&=[\Phi ,\Psi H^{-1} ].\label{eq:eta_compatibility_Pearson_1b}
	\end{align}
\end{subequations}
Relations \eqref{eq:eta_compatibility_Pearson_1a} and \eqref{eq:eta_compatibility_Pearson_1b} are \emph{gauge} equivalent.
\end{pro}

\section{The  ${}_1 F_2$ hypergeometrical family}
We choose $ \sigma(z)=\eta (z+a_1)$ and $\theta(z)=z(z+b_1)(z+b_2)$
with corresponding weight given by
\begin{align*}
	w(z)=\frac{(a_1)_z}{(b_1+1)_z(b_2+1)_z}\frac{\eta^z}{z!}.
\end{align*}
with  first moment given by the  generalized  hypergeometric function:
$	\rho_0={}_1F_2\left[\!\!{\begin{array}{c}a_1\\b_1+1,b_2+1\end{array}};\eta\right]$, and subsequent moments
$\rho_n=\vartheta_{\eta}^n {}_1F_2\left[\!\!{\begin{array}{c}a_1\\b_1+1,b_2+1\end{array}};\eta\right]$. The corresponding moments exist for any $\eta\in\C$.

\begin{theorem}[The generalized Laguerre--Freud  structure matrix]\label{teo:1F2}
	For $\sigma(z)=\eta(z+a_1)$ and  $\theta(z)=z(z+b_1)(z+b_2)$  
	we find for the subleading coefficients the following expression
	\begin{multline*}p^1_n=\frac{n(n+1)}{2}-n\beta_n-\frac{1}{\eta+\gamma_{n+1}}(\gamma_{n+1}(\gamma_n+\gamma_{n+1}+\gamma_{n+2}+(\beta_{n+1}+b_1)(\beta_{n+1}+b_2)\\+\beta_n(\beta_{n+1}+\beta_n+b_1+b_2)-n(2\beta_n+\beta_{n+1}+b_1+b_2-n)-\eta)-\eta(n+1)(\beta_n+a_1))
	\end{multline*}
and we also find
\begin{multline*}
	\pi^{[2]}_{n-1}=\frac{1}{\eta+\gamma_{n+1}}\Big(
\gamma_{n+1}\big(
\gamma_n+\gamma_{n+1}+\gamma_{n+2}+(\beta_{n+1}+b_1)(\beta_{n+1}+b_2)
+\beta_{n}(\beta_{n+1}+\beta_{n}+b_1+b_2)\\-n(2\beta_n+\beta_{n+1}+b_1+b_2-2n)-\eta
\big)-\eta(n+1)(a_1+a_2)
\Big).
\end{multline*}

	The Laguerre--Freud structure matrix is the following pentadiagonal matrix
	\begin{align*}
		\Psi=\left[\begin{NiceMatrix}[
			small]
			\psi^{(0)}_0 &\psi^{(1)}_0 &\psi^{(2)}_0&H_3&0 &\Cdots&&\\
			\eta H_1& \psi^{(0)}_1 &\psi^{(1)}_1&\psi^{(2)}_1& H_4&\Ddots&&\\[5pt]
			0& \eta H_2&\psi^{(0)}_2&\psi^{(1)}_2& \psi^{(2)}_2& H_5&&\\
			\Vdots&\Ddots &\Ddots&\Ddots&\Ddots&\Ddots&\Ddots&\Ddots\\
		\end{NiceMatrix}\right],
	\end{align*}
	with
	\begin{align*}
		\psi^{(0)}_n&=\eta H_n(n+\beta_{n}+a_1),\\
		\psi^{(1)}_n&=\eta\Big(H_{n+1}+H_n\big(\pi^{[2]}_{n-1}+(n+1)(\beta_n+a_1)\big)\Big),\\
		\psi^{(2)}_n&=H_{n+2}(\beta_n+\beta_{n+1}+\beta_{n+2}+b_1+b_2-n).
	\end{align*}

\end{theorem}

\begin{proof}

As $\deg \sigma=1$ and  $\deg \theta=3$ for the Lagurre--Freud matrix we have
\begin{align*}
	\Psi=\Lambda^\top\psi^{(-1)}+\psi^{(0)}+\psi^{(1)}\Lambda+\psi^{(2)}\Lambda^2+\psi^{(3)}\Lambda^3.
\end{align*}

 Form $\Psi=\sigma(J)H\Pi^\top$  we get the first subdiagonal, the main diagonal and an expression for the first superdiagonal depending upon 
  $\pi^{[2]}$:
\begin{align*}
	\psi^{(-1)}&= \eta T_{-}H ,\\ \psi^{(0)}&= \eta H(T_{+}D+\beta+a_1),\\ \psi^{(1)}&=\eta\Big(T_{-}H+H\big(T_{+}\pi^{[2]}+D(\beta+a_1)\big)\Big).
\end{align*}

From $\Psi=\Pi^{-1}H\theta(J^\top)$ we get an expression for the first superdiagonal depending on de $T_{+}^2\pi^{[-2]}$and the other two nonzero superdiagonals
\begin{align*}
\psi^{(1)}&=	\begin{multlined}[t][.9\textwidth]T_{-}H[T_{+}\gamma+\gamma+T_{-}\gamma+(T_{-}\beta+b_1)(T_{-}\beta+b_2)+\beta(T_{-}\beta+\beta+b_1+b_2)-T_{+}D(T_{+}\beta+\beta\\+b_1+b_2+T_{-}\beta)+T_{+}^2\pi^{[-2]}],
\end{multlined}\\
\psi^{(2)}&=T_{-}^2H[\beta+T_{-}\beta+T_{-}^2\beta+b_1+b_2-T_{+}D],\\ \psi^{(3)}&=T_{-}^3H.
\end{align*}

To obtain the diagonal matrix  $\pi^{[\pm2]}$ we combine  $T^{2}_{+}\pi^{[-2]}-T_{+}\pi^{[-2]}=T_{+}D(T_{+}\beta-\beta-1)$ with  $\pi^{[-2]}=2D^{[2]}-\pi^{[2]}$ to get
\begin{align*}
	T_{+}^2\pi^{[-2]}=T_{+}D( T_{+}\beta-\beta-1)+2T_{+}D^{[2]}-T_{+}\pi^{[2]}.
\end{align*}
Now, equating the two formulas for $\psi^{(1)}$ and using the previous results we find
\begin{multline*}T_{+}\pi^{[2]}(\eta H+T_{-}H)=T_{-}H[T_{+}\gamma+\gamma+T_{-}\gamma+(T_{-}\beta+b_1)(T_{-}\beta+b_2)+\beta(T_{-}\beta+\beta+b_1+b_2)\\-T_{+}D(\beta+T_{-}\beta+b_1+b_2)-T_{+}D(\beta+1)+2T_{+}D^{[2]}-\eta I]-\eta H D(\beta+a_1).
\end{multline*}
Consequently, we derive the given expression for  $p^1_n$.
\end{proof}

We now explore the consequences of the compatibility relation $[\Psi H^{-1},J]=\Psi H^{-1}$. 
We get trivial relations and
$T_{+}\pi^{[2]}-T_{+}^2\pi^{[2]}=T_{+}D(T_{+}\beta-\beta+I)$. However, a new
 formula appears from the first superdiagonal:
\begin{multline*}\eta(\beta+T_{-}\beta-1)=T_{-}\gamma(T_{+}D-2T_{-}\beta-T_{-}^2\beta+b_1+b_2+1)+\gamma(\beta+1-T_{-}\beta)\\+T_{+}\gamma(T_{+}\beta+2\beta+b_1+b_2-T_{+}^2D+1)+\beta^3-(T_{-}\beta)^3+(T_{-}\beta)^2+\beta^2+\beta T_{-}\beta+\\(b_1+b_2)(\beta^2-(T_{-}\beta)^2+T_{-}\beta+\beta)-T_{+}D(2\beta^2-\beta T_{-}\beta-(T_{-}\beta)^2+3\beta)\\+(\beta+1-T_{-}\beta)[2T_{+}D^{[2]}-T_{+}D(b_1+b_2)+b_1b_2-T_{+}\pi^{[2]}].
\end{multline*}
No new  relations are obtained form the second compatibility \eqref{eq:eta_compatibility_Pearson_1a}.
Using this compatibility, we seek for  Laguerre--Freud equations
\begin{align}\label{eq:LF}
	\beta_{n+1}&=\mathscr B_n(\gamma_{n+1},\gamma_{n},\dots,\beta_{n},\beta_{n-1},\dots),&
	\gamma_{n+2}&=\mathscr G_n(\gamma_{n+1},\gamma_{n},\dots,\beta_{n+1},\beta_n,\dots).
\end{align}

\begin{theorem}
We find the following Laguerre--Freud equations
\begin{align*}
	\gamma_{n+2}&=\begin{multlined}[t][.9\textwidth]\frac{\eta+\gamma_{n+1}}{\gamma_{n+1}}n(\beta_{n-1}-\beta_n+1)-\big(\gamma_n+\gamma_{n+1}+(\beta_{n+1}+b_1)(\beta_{n+1}+b_2)\\+\beta_n(\beta_{n+1}+\beta_n+b_1+b_2)-n(2\beta_n+\beta_{n+1}+b_1+b_2-n)-\eta\big)+\frac{\eta}{\gamma_{n+1}}(n+1)(\beta_n+a_1)\\+\frac{\eta+\gamma_{n+1}}{(\eta+\gamma_n)\gamma_{n+1}}\big(\gamma_n(\gamma_{n-1}+\gamma_n+\gamma_{n+1}+(\beta_n+b_1)(\beta_n+b_2)+\beta_{n-1}(\beta_n+\beta_{n-1}+b_1+b_2)\\-(n-1)(2\beta_{n-1}+\beta_n+b_1+b_2-n+1)-\eta)-\eta n(\beta_{n-1}+a_1)\big),\end{multlined}\\
\beta_{n+2}&=\begin{multlined}[t][.8\textwidth]n-2\beta_{n+1}+b_1+b_2+1+\frac{1}{\gamma_{n+2}}\Big(\eta(1-\beta_n-\beta_{n+1})+\gamma_{n+1}(\beta_n+1-\beta_{n+1})\\+\gamma_n(\beta_{n-1}+2\beta_n+b_1+b_2-n+2)+\beta_n^3-\beta_{n+1}^3+\beta_{n+1}^2+\beta_n^2+\beta_n\beta_{n+1}\\+(b_1+b_2)(\beta_n^2-\beta^2_{n+1}+\beta_{n+1}+\beta_n)-n(2\beta_n^2-\beta_n\beta_{n+1}-\beta_{n+1}^2+3\beta_n)\\+(\beta_n-\beta_{n+1}+1)\big(n(n+1)-n(b_1+b_2)+b_1b_2\\-\frac{1}{\eta+\gamma_{n+1}}(\gamma_{n+1}(\gamma_n+\gamma_{n+1}+\gamma_{n+2}+(\beta_{n+1}+b_1)(\beta_{n+1}+b_2)\\+\beta_n(\beta_{n+1}+\beta_n+b_1+b_2)-n(2\beta_n+\beta_{n+1}+b_1+b_2-n)-\eta)-\eta(n+1)(\beta_n+a_1))\big)\Big).\end{multlined}
\end{align*}
\end{theorem}
\begin{proof}
	To get $\gamma_{n+2}=F_1(n,\gamma_{n+1},\gamma_n,\dots,\beta_{n+1},\beta_n,\dots)$ we consider the compatibility  $T_{+}\pi^{[2]}-T_{+}^2\pi^{[2]}=T_{+}D(T_{+}\beta+1-\beta)$ and solve  $\gamma_{n+2}$. Also,
we obtain $\beta_{n+2}=F_2(n,\gamma_{n+2},\gamma_{n+1},\dots,\beta_{n+1},\beta_n,\dots)$ by solving for $\beta_{n+2}$ in the first subdiagonal
of the first type of compatibility.
\end{proof}

\section{The  ${}_2 F_2$ hypergeometrical family}
Let us take $\sigma(z)=\eta(z+a_1)(z+a_2)$ and  $\theta(z)=z(z+b_1)(z+b_2)$ and consider the corresponding 
Pearson equation 
\begin{align*}
	(k+1)(k+1+b_1)(k+1+b_2)w(k+1)=\eta(k+a_1)(k+a_2)w(k)
\end{align*}
whose solutions are proportional to 
$w(z)=\frac{(a_1)_z(a_2)_z}{(b_1+1)_z(b_2+1)_z}\frac{\eta^z}{z!}$, so that he discrete orthogonality measure is 
\begin{align*}
\rho_z=\sum_{k=0}^\infty \frac{(a_1)_k(a_2)_k}{(b_1+1)_k(b_2+1)_k}\frac{\eta^k}{k!},
\end{align*}
with  first moment given by the  generalized  hypergeometric function:
$	\rho_0={}_2F_2\left[\!\!{\begin{array}{c}a_1,a_2 \\b_1+1,b_2+1\end{array}};\eta\right]$, and subsequent moments
$\rho_n=\vartheta_{\eta}^n {}_2F_2\left[\!\!{\begin{array}{c}a_1,a_2 \\b_1+1,b_2+1\end{array}};\eta\right]$. The corresponding moments exist for any $\eta\in\C$.
\begin{definition}
	Let us introduce 
	\begin{align*}
	\hspace*{-1cm}
	A_n:=\begin{multlined}[t][\textwidth]
	-\eta^2\Big(\frac{n(n+1)}{2}+\beta_{n-1}+\beta_n+\gamma_{n+2}+\gamma_{n+1}+(n+1)(\beta_{n+1}+a_1+a_2)+(\beta_{n+1}+a_1)(\beta_{n+1}+a_2)\Big)\\+\eta\Big(4\gamma_n+\gamma_{n+1}(a_1+a_2-b_1-b_2+2(n+1)+\beta_{n-1}-\beta_{n+1})+\gamma_{n+2}(n-1-2\beta_{n+1}-\beta_{n+2}-b_1-b_2)\\+\frac{n(n+1)}{2}(a_1+a_2-b_1-b_2-\beta_{n+1})+(\beta_{n+1}+b_1+b_2
	+a_1+a_2-n^2)(\beta_n+\beta_{n-1})+2(\beta^2_n+\beta^2_{n-1})\\-(\beta_{n+1}+b_1)(\beta_{n+1}+b_2)(\beta_{n+1}-n-1)\Big)\\-\gamma_{n+1}\Big(\frac{n(n-1)}{2}+\gamma_{n+2}+\gamma_{n+1}+\gamma_n+\beta_{n-1}+(\beta_{n+1}+b_1)(\beta_{n+1}+b_2)+(\beta_n-n)(\beta_{n+1}+\beta_n+b_1+b_2)\Big),
	\end{multlined}
	\end{align*}
	and
	\begin{align}	
	\label{eq:Bn}	B_n&:=\eta^2+\eta(2\beta_{n-1}+2\beta_n+\beta_{n+1}+a_1+a_2+b_1+b_2+2n)+\gamma_{n+1}.
	\end{align}
\end{definition}

\begin{theorem}[The generalized Laguerre--Freud  structure matrix]\label{teo:Hahn}
	For $\sigma(z)=\eta(z+a_1)(z+a_2)$ and  $\theta(z)=z(z+c_1)(z+c_2)$  
	we find for the subleading coefficients the following expression
\begin{align}\label{eq:pn}
	p^1_{n-1}&=\frac{A_n(\eta,\beta_{n-1},\beta_n,\beta_{n+1},\beta_{n+2},\gamma_n,\gamma_{n+1},\gamma_{n+2})}{B_n(\eta,\beta_{n-1},\beta_n,\beta_{n+1},\gamma_{n+1})}.
\end{align}

The Laguerre--Freud structure matrix is the following hexadiagonal matrix
\begin{align*}
	\Psi=\left[\begin{NiceMatrix}[
		small]
		\psi^{(0)}_0 &\psi^{(1)}_0 &\psi^{(2)}_0&H_3&0 &\Cdots&&\\
		\psi^{(-1)}_0 & \psi^{(0)}_1 &\psi^{(1)}_1&\psi^{(2)}_1& H_4&\Ddots&&\\[4pt]
		\eta H_2& \psi^{(-1)}_1 &\psi^{(0)}_2&\psi^{(1)}_2& \psi^{(2)}_2& H_5&&\\[4pt]
		0&	\eta H_ 3& \psi^{(-1)}_2 &\psi^{(0)}_3&\psi^{(1)}_3& \psi^{(2)}_3& H_6&\\
		\Vdots&\Ddots &\Ddots&\Ddots&\Ddots&\Ddots&\Ddots&\Ddots\\
	\end{NiceMatrix}\right],
	\end{align*}
with
\begin{align*}
	\psi^{(2)}_n&=(\beta_n+\beta_{n+1}+\beta_{n+2}+b_1+b_2-n)H_{n+2},\\
	\psi^{(0)}_n&=\eta \Big(\frac{n(n-1)}{2}+\beta_{n-1}+n(\beta_n+a_1+a_2)+\gamma_n+\gamma_{n+1}+(\beta_n+a_1)(\beta_n+a_2)-p^1_{n-1}\Big)H_n,\\
\psi^{(1)}_n&=\begin{multlined}[t][.8\textwidth]
\Big(\frac{n(n-1)}{2}+\gamma_n+\gamma_{n+1}+\gamma_{n+2}+(\beta_{n+1}+b_1)(\beta_{n+1}+b_2)\\+(\beta_n-n)(\beta_n+\beta_{n+1}+b_1+b_2)-\beta_{n-1}+p^1_{n-1}\Big)H_{n+1},
\end{multlined}\\
	\psi^{(-1)}_n&=\eta (\beta_n+\beta_{n+1}+a_1+a_2+n)H_{n+1}.
\end{align*}
\end{theorem}

\begin{proof}
We  write \eqref{eq:Psi} and \eqref{eq:Psi2} as
\begin{align} 
	\label{eqref:fl1}\Psi=\sigma(J)H\Pi^{\top}, \\ 
	\label{eqref:fl2} \Psi=\Pi^{-1}H\theta(J^{\top}). 
\end{align}
The Laguerre--Freud matrix $\Psi$ is a hexadiagonal matrix
\begin{align*}
	\Psi=(\Lambda^{\top})^2\psi^{(-2)}+\Lambda^{\top}\psi^{(-1)}+\psi^{(0)}+\psi^{(1)}\Lambda+\psi^{(2)}\Lambda^2+\psi^{(3)}\Lambda^3,
\end{align*} with  $\psi^{(n)}$ diagonal matrices.
From Equation \eqref{eqref:fl1} we get
\begin{align*} \psi^{(-2)}&=\eta T^2_{-}H ,&
	 \psi^{(-1)}&=\eta T_{-}H(T_{+}D+\beta+T_{-}\beta+a_1+a_2).
 \end{align*}
From Equation  \eqref{eqref:fl2} we obtain
\begin{align*} 
	\psi^{(3)}&=T^3_{-}H, &
	\psi^{(2)}&=T_{-}^2H(T^2_{-}\beta+\beta+T_{-}\beta-T_{+}D+b_1+b_2).
\end{align*}

We also have expressions for the diagonal matrices $\psi^{(0)},\psi^{(1)}$ in terms of 
$\pi^{[\pm 2]},\pi^{[\pm3]}$. From  Equation \eqref{eqref:fl1} we get
\begin{align}\label{eqref:psi0a}\psi^{(0)}&= \eta H[T^2_{+}\pi^{[2]}+T_{+}D(T_{+}\beta+\beta+a_1+a_2)+\gamma+T_{+}\gamma+(\beta+a_1)(\beta+a_2)],\\
\label{eqref:psi1a}\psi^{(1)}&=\eta H[\gamma(\beta+T_{-}\beta+a_1+a_2)+D(\gamma+T_{+}\gamma+(\beta+a_1)(\beta+a_2))+T_{+}^2\pi^{[3]}+T_{+}\pi^{[2]}(T_{+}\beta+\beta+a_1+a_2)],\end{align}   
and Equation \eqref{eqref:fl2} leads to
\begin{align}\label{eqref:psi0b}
	\psi^{(0)}&=
	\begin{multlined}[t][0.8\textwidth]
	H[(\beta+b_1)(\beta+b_2)(\beta-T_{+}D)+\gamma(2\beta+T_{-}\beta+b_1+b_2-T_{+}D)+T_{+}\gamma(2\beta+T_{+}\beta+b_1+b_2-T_{+}D)\\-T_{+}D[T_{+}^2\gamma+T_{+}\beta(\beta+T_{+}\beta+b_1+b_2)]+T_{+}^2\pi^{[-2]}(\beta+T_{+}\beta+T_{+}^2\beta+b_1+b_2)-T_{+}^3\pi^{[-3]}],
	\end{multlined}\\
\label{eqref:psi1b}
\psi^{(1)}&=
\begin{multlined}[t][0.9\textwidth]
	T_{-}H[T_{+}^2\pi^{[-2]}-T_{+}D(T_{-}\beta+\beta+T_{+}\beta+b_1+b_2)+T_{-}\gamma+\gamma+T_{+}\gamma+(T_{-}\beta+b_1)(T_{-}\beta+b_2)\\+\beta(T_{-}\beta+\beta+b_1+b_2)].
\end{multlined}\end{align}
We get a system of two equations by equating the RHS of \eqref{eqref:psi0a} and \eqref{eqref:psi0b} and the RHS of \eqref{eqref:psi1a} and  \eqref{eqref:psi1b}. Using the shift   $n \rightarrow n+1$ in the first relation obtained   from the two possible expressions for  $\psi^{(0)}$,  we now  show that the $\pi$ matrices involved in the system 
can be expressed uniquely in terms of $p^2_n,p^1_{n-1}$.
From the relations
\begin{align} 
	p^1_{n}&=p^1_{n-1}-\beta_{n-1} ,\\ p^1_{n+1}&=p^1_{n-1}-\beta_{n-1}-\beta_n, \\ p^2_{n+1}&=p^2_n-\gamma_n-\beta_np^1_{n-1}+\beta_n\beta_{n-1}, 
\end{align}
we get 
\begin{align*}\pi^{[\pm 2]}_{n-1}&=\frac{n(n+1)}{2}\mp( n\beta_n+p^1_{n-1}-\beta_{n-1}),\\\pi^{[\pm 2]}_{n-2}&=\frac{n(n-1)}{2}\mp(n-1)\beta_{n-1}\mp p^1_{n-1},\\
	\pi^{[\pm3]}_{n-2}&=
	\begin{multlined}[t][.8\textwidth]
	\pm\frac{n(n+1)(n-1)}{3}-\frac{n(n-1)}{2}(\beta_{n-1}+\beta_n)\pm(n-1)(\beta_n\beta_{n-1}-\gamma_n)\\+p^1_{n-1}(\pm\beta_n\mp\beta_{n-1}-n)+\pm(p^1_{n-1})^2\mp2p^2_n.
	\end{multlined}
\end{align*}

Then, we  solve the equation for  $p^2_n$ in the equation obtained for the two expressions \eqref{eqref:psi0a} and \eqref{eqref:psi0b} for  $\psi^{(0)}$ and substitute in the Equation gotten from the two expressions \eqref{eqref:psi1a} and \eqref{eqref:psi1b}  for $\psi^{(1)}$ to deduce
\begin{multline}\label{p1}
p^1_{n-1}
\big(\eta^2+\eta(2\beta_{n-1}+2\beta_n+\beta_{n+1}+a_1+a_2+b_1+b_2+2n)+\gamma_{n+1}
\big)\\=-\eta^2
\Big(\frac{n(n+1)}{2}+\beta_{n-1}+\beta_n+\gamma_{n+2}+\gamma_{n+1}+(n+1)(\beta_{n+1}+a_1+a_2)\\
+(\beta_{n+1}+a_1)(\beta_{n+1}+a_2)\Big)\\+\eta
\Big(\gamma_{n+1}(a_1+a_2-b_1-b_2+2(n+1)+\beta_{n-1}-\beta_{n+1})+\gamma_{n+2}(n-1-2\beta_{n+1}-\beta_{n+2}-b_1-b_2)\\+n(\beta_{n-1}+\beta_n)+(a_1+a_2)\beta_{n-1}+4\gamma_n+2(\beta^2_n+\beta^2_{n-1})+\beta_n(\beta_{n+1}+a_1+a_2+b_1+b_2)\\
+(n+1)a_1a_2+\beta_{n-1}(\beta_{n+1}+b_1+b_2)-(\beta_{n+1}+b_1)
(\beta_{n+1}+b_2)(\beta_{n+1}-n-1)\\
+\frac{n(n+1)}{2}(a_1+a_2-b_1-b_2-\beta_{n+1})-n^2(\beta_n+\beta_{n-1})\Big)
\\-\gamma_{n+1}
\Big(\frac{n(n-1)}{2}+\gamma_{n+2}+\gamma_{n+1}+\gamma_n+\beta_{n-1}-n(\beta_{n+1}+\beta_n+b_1+b_2)+(\beta_{n+1}+b_1)(\beta_{n+1}+b_2)\\+\beta_n(\beta_{n+1}+\beta_n+b_1+b_2)\Big).
\end{multline}
That is, we have proven Equation \eqref{eq:pn} for $p^1_n$.
It is not necessary at this point to find  $p^2_n$ in order to get the Freud--Laguerre matrix, as we can get $\psi^{(0)}$ from \eqref{eqref:psi0a} and $\psi^{(1)}$ from \eqref{eqref:psi1b}. The entries in the diagonal matrices $\psi^{(0)},\psi^{(1)}$ can be written in terms of $p^1_{n-1}$ as:
\begin{align*}
	\psi^{(0)}_n&=\eta H_n\Big(\frac{n(n-1)}{2}+\beta_{n-1}+n(\beta_n+a_1+a_2)+\gamma_n+\gamma_{n+1}+(\beta_n+a_1)(\beta_n+a_2)-p^1_{n-1}\Big),\\
\psi^{(1)}_n&=
\begin{multlined}[t][.8\textwidth]
\gamma_{n+1}H_n\Big(\frac{n(n-1)}{2}+\gamma_{n+1}+\gamma_{n+2}+\gamma_n+(\beta_{n+1}+b_1)(\beta_{n+1}+b_2)\\+(\beta_n-n)(\beta_n+\beta_{n+1}+b_1+b_2)-\beta_{n-1}+p^1_{n-1}\Big).
\end{multlined}
\end{align*}
\end{proof}

\begin{definition}
	Let us consider	
	\begin{align}\label{eq:C}
	C_n&:=\begin{multlined}[t][.95\textwidth]
	\eta\big(
	(n+\beta_n+\beta_{n+1}+a_1+a_2)(\beta_n-\beta_{n+1}-1)-\beta_n-\beta_{n+1}+\gamma_n\big)\\-n(\beta_{n+1}-\beta_n-1)(\beta_{n+1}+\beta_n+\beta_{n-1})
	-\beta_n(\beta_{n+1}+\beta_n+b_1+b_2)\\-\gamma_n(\beta_{n-1}+2\beta_n+b_1+b_2-n+2)+\gamma_{n+2}(2\beta_{n+1}+b_1+b_2-n-1),
	\end{multlined}\\\label{eq:D}
	D_n&:=\frac{n(n-1)}{2}+\beta^2_{n+1}+(\beta_{n+1}-n)(b_1+b_2)+b_1b_2
	+(n-1)\beta_{n-1}+\gamma_{n+1},
	\end{align}
	and
		\begin{multline}
	\label{eq:E}	E_n:=-\eta\Big(\frac{n(n-1)}{2}+\beta_{n-1}+n(\beta_n+a_1+a_2)+(\beta_n+a_1)(\beta_n+a_2)\\
	-\gamma_n(n-2+\beta_{n-1}+\beta_n+a_1+a_2)+\gamma_{n+1}(n+1+\beta_n+\beta_{n+1}+a_1+a_2)\Big)\\+
	\gamma_{n+1}\Big(\frac{n(n-1)}{2}+\gamma_{n+1}-\beta_{n-1}+(\beta_n-n)(\beta_{n+1}+\beta_n+b_1+b_2)+(\beta_{n+1}+b_1)(\beta_{n+1}+b_2)\Big)\\
	-\gamma_n\Big(\frac{(n-1)(n-2)}{2}+\gamma_n+\gamma_{n-1}+(\beta_n+b_1)(\beta_n+b_2)+(\beta_{n-1}-n+1)(\beta_n+\beta_{n-1}+b_1+b_2)\Big)
	\end{multline}
\end{definition}

\begin{theorem}[Consequences of compatibility conditions]
	The compatibility conditions lead to   two alternative relations expressions for the subleading coefficient $p^1_{n-1}$ of the orthogonal polynomials:
	\begin{align}\label{eq:compatibility_4_1}
	p^1_{n-1}&=\frac{
		C_n(\eta,\beta_{n-1},\beta_n,\beta_{n+1},\gamma_n)
		+\gamma_{n+2}(\beta_{n+2}-\eta)
	}{1+\beta_n-\beta_{n+1}}-D_n(\beta_{n-1},\beta_{n+1},\gamma_{n+1}),\\
\label{eq:compatibility_4_2}
p^1_{n-1}&=\frac{E_n(\eta, \beta_{n-1},\beta_n,\beta_{n+1},\gamma_{n-1},\gamma_n,\gamma_{n+1})+\gamma_{n+1}\gamma_{n+2}}{\gamma_n-\gamma_{n+1}-\eta}.
\end{align}

	where $p^{1}_{n}$ is given in Equation \eqref{eq:pn}.
\end{theorem}

\begin{proof}
	We compute explicitly the compatibility equation $[\Psi H^{-1},J]=\Psi H^{-1}$. In the one hand $\Psi H^{-1}$
	can be expressed as
	\begin{align*}
		\Psi H^{-1}=(\Lambda^{\top})^2M^{(-2)}+\Lambda^{\top}M^{(-1)}+M^{(0)}+M^{(1)}\Lambda+M^{(2)}\Lambda^2+M^{(3)}\Lambda^3
	\end{align*}
with $M^{(n)}$ being the following diagonal matrices
	\begin{align*}
		M^{(3)}&=I,\\
M^{(2)}&=T_{-}^2\beta+T_{-}\beta+\beta-T_{+}D+b_1+b_2,\\
M^{(1)}&=\begin{multlined}[t][.9\textwidth]
T_{+}^2\pi^{[-2]}-T_{+}D(T_{-}\beta+\beta+T_{+}\beta+b_1+b_2)+T_{-}\gamma+\gamma+T_{+}\gamma+(T_{-}\beta+b_1)(T_{-}\beta+b_2)\\+\beta(T_{-}\beta+\beta+b_1+b_2),
\end{multlined}\\
M^{(0)}&=\eta\Big(T_{+}^2\pi^{[2]}+T_{+}D(T_{+}\beta+\beta+a_1+a_2)+\gamma+T_{+}\gamma+(\beta+a_1)(\beta+a_2)\Big),
\end{align*}
or alternatively
\begin{align*}
	M^{(0)}&=\begin{multlined}[t][.9\textwidth]\beta^3+(b_1+b_2)\beta^2+b_1b_2\beta-T_{+}D[\beta^2+(b_1+b_2)\beta+b_1b_2+T_{+}\beta(\beta+T_{+}\beta+b_1+b_2)]\\+\gamma(2\beta+T_{-}\beta+b_1+b_2-T_{+}D)+T_{+}\gamma(2\beta+T_{+}\beta+b_1+b_2-T_{+}D)\\-T_{+}^2\gamma T_{+}D+T_{+}^2\pi^{[-2]}(\beta+T_{+}\beta+T_{+}^2\beta+b_1+b_2)-T_{+}^3\pi^{[-3]},\end{multlined}\\
M^{(-1)}&=\eta\gamma(T_{+}D+\beta+T_{-}\beta+a_1+a_2),\\
M^{(-2)}&=\eta\gamma( T_{-}\gamma).
\end{align*}
In the other hand, the commutator can be written as follows
	\begin{align*}[\Psi H^{-1},J]=(\Lambda^{\top})^2\tilde{M}^{(-2)}+\Lambda^{\top}\tilde{M}^{(-1)}+\tilde{M}^{(0)}+\tilde{M}^{(1)}\Lambda+\tilde{M}^{(2)}\Lambda^2+\tilde{M}^{(3)}\Lambda^3,
	\end{align*}
where the diagonal matrices $\tilde{M}^{(n)}$ are
	\begin{align*}\tilde{M}^{(3)}&=I,\\
\tilde{M}^{(2)}&=T_{+}^2\pi^{[-2]}-T_{+}\pi^{[-2]}+T_{+}D(\beta-T_{+}\beta)+T_{-}^2\beta+T_{-}\beta+\beta+b_1+b_2,\\
\tilde{M}^{(1)}&=\begin{multlined}[t][.8\textwidth]\eta\big(T_{+}^2\pi^{[2]}-T_{+}\pi^{[2]}-(\beta+a_1+a_2)+T_{+}(D\beta)-DT_{-}\beta\big)\\+(T_{-}\beta-\beta)\big(T_{+}^2\pi^{[-2]}-T_{+}D(T_{+}\beta+b_1+b_2)+
	\gamma+(T_{-}\beta+b_1)(T_{-}\beta+b_2)
	-\eta(a_1+a_2)\big)\\+(\beta^2-(T_{-}\beta)^2)(T_{+}D+\eta)+T_{-}\gamma(T_{-}^2\beta+2T_{-}\beta-T_{+}D+b_1+b_2-\eta)
	\\-T_{+}\gamma(2\beta+T_{+}\beta-T_{+}^2D+b_1+b_2-\eta),\end{multlined}\\
\tilde{M}^{(0)}&=\begin{multlined}[t][.85\textwidth]\eta T_{+}\gamma\big(T_{+}^2D+T_{+}\beta+\beta+a_1+a_2\big)\\+\gamma\big(T_{+}^2\pi^{[-2]}-T_{+}D(T_{-}\beta+\beta+T_{+}\beta+b_1+b_2)
	+T_{-}\gamma+\gamma+(T_{-}\beta+b_1)(T_{-}\beta+b_2)
\\	+\beta(T_{-}\beta+\beta+b_1+b_2)\big)-T_{+}\gamma\big(T_{+}^3\pi^{[-2]}-T_{+}^2D(\beta+T_{+}\beta+T_{+}^2\beta+b_1+b_2)\\+T_{+}\gamma+T_{+}^2\gamma+(\beta+b_1)(\beta+b_2)\T_{+}\beta(\beta+T_{+}\beta+b_1+b_2)-\eta\gamma(T_{+}D+\beta+T_{-}\beta+a_1+a_2),\end{multlined}\\
\tilde{M}^{(-1)}&=\eta\gamma(T_{+}\pi^{[2]}-T_{+}^2\pi^{[2]}+T_{+}D(-T_{-}\beta-T_{+}\beta)+D(\beta+T_{-}\beta)+(a_1+a_2)),\\
\tilde{M}^{(-2)}&=\eta\gamma(T_{-}\gamma).
\end{align*}
	From the equations derived from second superdiagonal and first  subdiagonal one obtains, respectively, the following equations
	\begin{align*}T^2_{+}\pi^{[-2]}-T_{+}\pi^{[-2]}&=T_{+}D(T_{+}\beta-\beta-1),&
		T_{+}\pi^{[2]}-T_{+}^2\pi^{[2]}&=T_{+}D(1-\beta+T_{+}\beta).\end{align*}
	Hence , they can be gotten directly from the expressions for the diagonal matrices $\pi^{[\pm2]}$.
	
From the first superdiagonal we get \eqref{eq:compatibility_4_1}
and using the alternative expression of the main diagonal it can be seen that it corresponds to the identity: 
	\begin{align*}
		\pi^{[-3]}_{n-2}-\pi^{[-3]}_{n-3}=\pi^{[-2]}_{n-2}(1+\beta_n-\beta_{n-2})-(n-1)\gamma_n+n\gamma_{n-1}.
	\end{align*}
	Compatibility in the main diagonal is \eqref{eq:compatibility_4_2}.
\end{proof}

%

We discuss now about Laguerre--Freud equations \eqref{eq:LF}.

\begin{definition}
	Let us introduce 
\begin{align}\label{eq:hatA}
\hat A_n&:=	\begin{multlined}[t][.9\textwidth]	-\eta^2\Big(\frac{n(n+1)}{2}+\beta_{n-1}+\beta_n+\gamma_{n+1}+(n+1)(\beta_{n+1}+a_1+a_2)+(\beta_{n+1}+a_1)(\beta_{n+1}+a_2)\Big)\\
+\eta\Big(4\gamma_n+\gamma_{n+1}(a_1+a_2-b_1-b_2+2(n+1)+\beta_{n-1}-\beta_{n+1})\\+\frac{n(n+1)}{2}(a_1+a_2-b_1-b_2-\beta_{n+1})+(\beta_{n+1}+b_1+b_2
+a_1+a_2-n^2)(\beta_n+\beta_{n-1})+2(\beta^2_n+\beta^2_{n-1})\\-(\beta_{n+1}+b_1)(\beta_{n+1}+b_2)(\beta_{n+1}-n-1)\Big)\\-\gamma_{n+1}\Big(\frac{n(n-1)}{2}+\gamma_{n+1}+\gamma_n+\beta_{n-1}+(\beta_{n+1}+b_1)(\beta_{n+1}+b_2)+(\beta_n-n)(\beta_{n+1}+\beta_n+b_1+b_2)\Big),
\end{multlined}\\\label{eq:F}
F_n&:=-\frac{E_n}{\gamma_n-\gamma_{n+1}-\eta}
+\frac{C_n}{1+\beta_n-\beta_{n+1}}-D_n+
\frac{	1}{\eta(1+\beta_n-\beta_{n+1})}\Big(\hat A_n
-	\frac{E_n}{\gamma_n-\gamma_{n+1}-\eta}
B_n\Big),\\\label{eq:G}
G_n&:=\frac{\gamma_{n+1}}{\gamma_n-\gamma_{n+1}-\eta}+\frac{\eta}{1+\beta_n-\beta_{n+1}}+\frac{\eta^2-\eta(n-1-2\beta_{n+1}-b_1-b_2)+\gamma_{n+1}	+\frac{\gamma_{n+1}B_n}{\gamma_n-\gamma_{n+1}-\eta}}{\eta(1+\beta_n-\beta_{n+1})}.
\end{align}
\end{definition}

Then, we have the following result
\begin{theorem}[Laguerre--Freud equations]
Given the  functions $\hat A_n,B_n,C_n, D_n,E_n,F_n$ and $G_n$, as in \eqref{eq:hatA}, \eqref{eq:Bn},\eqref{eq:C},\eqref{eq:D},\eqref{eq:E},\eqref{eq:F} and \eqref{eq:G}, respectively,  the following nonlinear equations of Laguerre--Freud type are satisfied
	\begin{align}
	\label{eq:LFbeta}
\beta_{n+2}&=\begin{multlined}[t][\textwidth]\frac{	
\hat A_n(\eta,\beta_{n-1},\beta_n,\beta_{n+1},\gamma_n,\gamma_{n+1})-
(\eta^2-\eta(n-1-2\beta_{n+1}-b_1-b_2)+\gamma_{n+1})\gamma_{n+2}}{{\eta\gamma_{n+2}}}\\-	\frac{E_n(\eta, \beta_{n-1},\beta_n,\beta_{n+1},\gamma_{n-1},\gamma_n,\gamma_{n+1})+\gamma_{n+1}\gamma_{n+2}}{\gamma_n-\gamma_{n+1}-\eta}
\frac{	B_n(\eta,\beta_{n-1},\beta_n,\beta_{n+1},\gamma_{n+1})}{{\eta\gamma_{n+2}}},
\end{multlined}\\
\label{eq:LFgamma}
\gamma_{n+2}&=\frac{F_n(\eta,\beta_{n-1},\beta_n,\beta_{n+1},\gamma_{n-1},\gamma_n,\gamma_{n+1})}{G_n(\eta,\beta_{n-1},\beta_n,\beta_{n+1},\gamma_n,\gamma_{n+1})}.
	\end{align}
\end{theorem}

\begin{proof}
	From \eqref{eq:pn} we get
	\begin{align}
\beta_{n+2}=\frac{	\tilde A_n(\eta,\beta_{n-1},\beta_n,\beta_{n+1},\gamma_n,\gamma_{n+1},\gamma_{n+2})-	p^1_{n-1}B_n(\eta,\beta_{n-1},\beta_n,\beta_{n+1},\gamma_{n+1})
}{\eta\gamma_{n+2}},
	\end{align}
	with
	\begin{align*}
	\tilde 	A_n:=\hat A_n(\eta,\beta_{n-1},\beta_n,\beta_{n+1},\gamma_n,\gamma_{n+1})-
	(\eta^2-\eta(n-1-2\beta_{n+1}-b_1-b_2)+\gamma_{n+1})\gamma_{n+2}.
	\end{align*}
	Therefore, recalling \eqref{eq:compatibility_4_1}, we get the Laguerre--Freud  equation  \eqref{eq:LFbeta}.
	
To prove the second Laguerre--Freud equation \eqref{eq:LFgamma} we first	notice that  \eqref{eq:LFbeta} can be written as
		\begin{multline*}
2\eta\beta_{n+2}\gamma_{n+2}=\tilde A_n(\eta,\beta_{n-1},\beta_n,\beta_{n+1},\gamma_n,\gamma_{n+1},\gamma_{n+2})\\-	\frac{E_n(\eta, \beta_{n-1},\beta_n,\beta_{n+1},\gamma_{n-1},\gamma_n,\gamma_{n+1})+\gamma_{n+1}\gamma_{n+2}}{\gamma_n-\gamma_{n+1}-\eta}
B_n(\eta,\beta_{n-1},\beta_n,\beta_{n+1},\gamma_{n+1}).
\end{multline*}
Consequently,  recalling \eqref{eq:compatibility_4_1} we find
	\begin{align*}
p^1_{n-1}&=
\frac{C_n(\eta,\beta_{n-1},\beta_n,\beta_{n+1},\gamma_n)
	-\eta\gamma_{n+2}	}{1+\beta_n-\beta_{n+1}}
-D_n(\beta_{n-1},\beta_{n+1},\gamma_{n+1})+	
\frac{	\gamma_{n+2}\beta_{n+2}}{1+\beta_n-\beta_{n+1}}
\\
	&=
\begin{multlined}[t][.9\textwidth]
	\frac{C_n(\eta,\beta_{n-1},\beta_n,\beta_{n+1},\gamma_n)
		-\eta\gamma_{n+2}}{1+\beta_n-\beta_{n+1}}-D_n(\beta_{n-1},\beta_{n+1},\gamma_{n+1})\\+	
	\frac{	1}{\eta(1+\beta_n-\beta_{n+1})}\Big(\tilde A_n(\eta,\beta_{n-1},\beta_n,\beta_{n+1},\gamma_n,\gamma_{n+1},\gamma_{n+2})\\-	\frac{E_n(\eta, \beta_{n-1},\beta_n,\beta_{n+1},\gamma_{n-1},\gamma_n,\gamma_{n+1})+\gamma_{n+1}\gamma_{n+2}}{\gamma_n-\gamma_{n+1}-\eta}
	B_n(\eta,\beta_{n-1},\beta_n,\beta_{n+1},\gamma_{n+1})\Big).
\end{multlined}
	\end{align*}
Thus, we see that the following relation is fulfilled 
		\begin{align*}
	p^1_{n-1}&=
	\begin{multlined}[t][.9\textwidth]
	\frac{C_n
		-\eta\gamma_{n+2}}{1+\beta_n-\beta_{n+1}}-D_n+	
	\frac{	1}{\eta(1+\beta_n-\beta_{n+1})}\Big(\hat A_n-
	(\eta^2-\eta(n-1-2\beta_{n+1}-b_1-b_2)+\gamma_{n+1})\gamma_{n+2}\\-	\frac{E_n+\gamma_{n+1}\gamma_{n+2}}{\gamma_n-\gamma_{n+1}-\eta}
	B_n\Big).
	\end{multlined}
	\end{align*}
	Now, from Equation \eqref{eq:compatibility_4_1} we get
	\begin{align*}
\frac{E_n+\gamma_{n+1}\gamma_{n+2}}{\gamma_n-\gamma_{n+1}-\eta}
=
\begin{multlined}[t][.85\textwidth]
\frac{C_n
	-\eta\gamma_{n+2}}{1+\beta_n-\beta_{n+1}}-D_n\\+	
\frac{	1}{\eta(1+\beta_n-\beta_{n+1})}\Big(\hat A_n-
(\eta^2-\eta(n-1-2\beta_{n+1}-b_1-b_2)+\gamma_{n+1})\gamma_{n+2}\\-	\frac{E_n+\gamma_{n+1}\gamma_{n+2}}{\gamma_n-\gamma_{n+1}-\eta}
B_n\Big).
\end{multlined}
	\end{align*}
Hence, we find the second Laguerre--Freud equation \eqref{eq:LFgamma}.
\end{proof}

%
%
%

\setcounter{MaxMatrixCols}{30}  
\section{The  ${}_3F_2$ hypergeometrical case}
Let us take $\sigma(z)=\eta(z+a_1)(z+a_2)(z+a_3)$ and  $\theta(z)=z(z+b_1)(z+b_2)$ and consider the corresponding 
Pearson equation 
\begin{align*}
	(k+1)(k+1+b_1)(k+1+b_2)w(k+1)=\eta(k+a_1)(k+a_2)(k+a_3)w(k)
\end{align*}
whose solutions are proportional to 
$w(z)=\frac{(a_1)_z(a_2)_z(a_3)_z}{(b_1+1)_z(b_2+1)_z}\frac{\eta^z}{z!}$,
so that he discrete orthogonality measure is 
\begin{align*}
\rho_z=\sum_{k=0}^\infty \frac{(a_1)_k(a_2)_k(a_3)_k}{(b_1+1)_k(b_2+1)_k}\frac{\eta^k}{k!},
\end{align*}
with  first moment given by the  generalized  hypergeometric function:
$	\rho_0={}_3F_2\left[\!\!{\begin{array}{c}a_1,a_2,a_3 \\b_1+1,b_2+1\end{array}};\eta\right]$,  and subsequent moments
$\rho_n=\vartheta_{\eta}^n {}_3F_2\left[\!\!{\begin{array}{c}a_1,a_2,a_3 \\b_1+1,b_2+1\end{array}};\eta\right]$. The corresponding moments exist for any $\eta\in\C$, whenever one of the $a$'s is a non-positive integer, when $|\eta|<1$   or when $|\eta|=1$ and
\begin{align*}
	 \operatorname{Re}(b_1+b_2-a_1-a_2-a_3)>0.
\end{align*}

\begin{definition}
	We consider
	\begin{align*}
		A_n&:=\begin{multlined}[t][.8\textwidth]n(n+1)+n(\beta_{n-1}+\beta_n+\beta_{n+1})+(a_1+a_2+a_3)(n+\beta_n+\beta_{n+1})+a_1a_2+a_2a_3+a_3a_1\\+\gamma_n+\gamma_{n+1}+\gamma_{n+2}+\beta_n^2+\beta_{n+1}^2+\beta_n\beta_{n+1},
		\end{multlined}\\
		B_n&=\begin{multlined}[t][.8\textwidth]
			-n(\beta_n^2-\beta_{n+1}^2+\beta_{n+1}+\beta_n-\gamma_{n+2})+(\beta_{n+1}-\beta_n-1)(n(b_1+b_2-\beta_{n-1})-b_1b_2)\\-(b_1+b_2)(\beta_{n+1}^2-\beta_n^2-\beta_n-\beta_{n+1})-\gamma_{n+2}(\beta_{n+2}+2\beta_{n+1}+b_1+b_2-1)-\gamma_{n+1}(\beta_{n+1}-\beta_n-1)\\+\gamma_n(\beta_{n-1}+2\beta_n+b_1+b_2-(n-2))-\beta_{n+1}^3+\beta_n^3+\beta_n^2+\beta_{n+1}^2+\beta_n\beta_{n+1},
		\end{multlined}\\
		C_n&=\eta+\beta_{n+1}-\beta_{n-2},\\
			D_n&=\begin{multlined}[t][.8\textwidth]
				\beta^3_n+(b_1+b_2)\left(\beta_n^2-\beta_{n-1}-\beta_{n-2}+\frac{n(n-1)}{2}\right)+b_1b_2\beta_n
				\\-n[\beta^2_n-\beta_{n-1}\beta_{n-2}+(b_1+b_2)\beta_n+\gamma_{n-1}]-\beta_n(\beta_{n-1}+\beta_{n-2})-2\beta_{n-1}\beta_{n-2}-\beta_{n-1}^2-\beta_{n-2}^2	\\+\gamma_n[b_1b_2+(b_1+b_2)(\beta_n+\beta_{n-1}-n+1)-n+\beta_{n-1}+2\beta_n+b_1+b_2]\\-\gamma_{n+1}[(b_1+b_2)(\beta_n+\beta_{n+1}-n)+b_1b_2+n-(\beta_{n+1}+2\beta_n+b_1+b_2)],
			\end{multlined}\\
			E_n&=n\gamma_{n+1}(\beta_n+\beta_{n+1})+\gamma_{n+1}(\beta_{n-1}+\beta_{n-2})-\gamma_n\beta_{n-2}-(n-1)\gamma_n(\beta_{n-1}+\beta_n),\\
			F_n&=\begin{multlined}[t][.95\textwidth]\gamma_{n+1}\left(\frac{n(n-1)}{2}+\gamma_{n+1}+\gamma_{n+2}+\beta_n^2+\beta_{n+1}^2+\beta_n\beta_{n+1}\right)\\-\gamma_n\left(\frac{(n-1)(n-2)}{2}+\gamma_{n-1}+\gamma_n+\beta_n^2+\beta_{n-1}^2+\beta_n\beta_{n-1}\right),
			\end{multlined}\\
			G_n&=\begin{multlined}[t][.8\textwidth](\gamma_{n+1}-\gamma_n)(a_1a_2+a_2a_3+a_3a_1+
				(a_1+a_2+a_3)(\beta_n+n+\gamma_{n+1}\beta_{n+1}-\gamma_n(\beta_{n-1}-1)).
			\end{multlined}
	\end{align*}
\end{definition}

\begin{theorem}(The Laguerre-Freud structure matrix)
The Laguerre-Freud structure matrix is the  following heptadiagonal matrix 
\begin{align*}
\Psi=\left[\begin{NiceMatrix}[
	small]
    \psi^{(0)}_0&\psi^{(1)}_0&\psi^{(2)}_0&H_3&0&\Cdots&&&&\\
    \psi^{(-1)}_0&\psi^{(0)}_1&\psi^{(1)}_1&\psi^{(2)}_1&H_4&\Ddots&&&&\\
    \psi^{(-2)}_0&\psi^{(-1)}_1&\psi^{(0)}_2&\psi^{(1)}_2&\psi^{(2)}_2&H_5&\Ddots&&&\\
    \eta H_3&\psi^{(-2)}_1&\psi^{(-1)}_2&\psi^{(0)}_3&\psi^{(1)}_3&\psi^{(2)}_3&H_6&\Ddots&&\\
    0&\eta H_4&\psi^{(-2)}_2&\psi^{(-1)}_3&\psi^{(0)}_4&\psi^{(1)}_4&\psi^{(2)}_4&H_7&\Ddots&\\
    \Vdots&\Ddots&\Ddots&\Ddots&\Ddots&\Ddots&\Ddots&\Ddots&\Ddots&\\
    \end{NiceMatrix}\right]
\end{align*}

where,
\begin{align*}\psi^{(-2)}_n&=\eta H_{n+2}(\beta_++\beta_{n+1}+\beta_{n+2}+a_1+a_2+a_3+n),\\
\psi^{(-1)}_n&=\begin{multlined}[t][\textwidth]\eta H_{n+1}(\pi^{[2]}_{n-2}+n(\beta_{n-1}+\beta_n+\beta_{n+1}+a_1+a_2+a_3)+\gamma_n+\gamma_{n+1}+\gamma_{n+2}+\beta^2_{n+1}+\beta^2_n\\+\beta_{n+1}\beta_n+(\beta_n+\beta_{n+1})(a_1+a_2+a_3)+a_1a_2+a_2a_3+a_3a_1),\end{multlined}\\
\psi^{(0)}&=\begin{multlined}[t][.9\textwidth]H_n(\beta_n(\beta_n+b_1)(\beta_n+b_2)+\gamma_n(\beta_{n-1}+2\beta_n+b_1+b_2)+\gamma_{n+1}(\beta_{n+1}+2\beta_n+b_1+b_2)\\-n(\beta_n^2+\beta^2_{n-1}+\beta_n\beta_{n-1}+(\beta_n+\beta_{n-1})(b_1+b_2)b_1b_2+\gamma_{n-1}+\gamma_n+\gamma_{n+1})\\+\pi^{[-2]}_{n-2}(\beta_n+\beta_{n-1}+\beta_{n-2}+b_1+b_2)-\pi^{[-3]}_{n-3}),\end{multlined}
\end{align*}
or, alternatively,

\begin{align*}
    \psi^{(0)}_n&=
    \begin{multlined}[t][.8\textwidth]\eta H_n(\pi^{[3]}_{n-3}+\pi^{[2]}_{n-2}(\beta_{n-2}+\beta_{n-1}+\beta_n+a_1+a_2+a_3)+n(\beta^2_{n-1}+\beta^2_n+\beta_n\beta_{n-1}+a_1a_2+a_2a_3+a_3a_1\\+(\beta_n+\beta_{n-1})(a_1+a_2+a_3)+\gamma_{n-1}+\gamma_n+\gamma_{n+1})+\gamma_n(\beta_{n-1}+2\beta_n+a_1+a_2+a_3)\\+\gamma_{n+1}(\beta_{n+1}+2\beta_n+a_1+a_2+a_3)+(\beta_n+a_1)(\beta_n+a_2)(\beta_n+a_3)),
\end{multlined}\\
\psi^{(1)}_n&=
\begin{multlined}[t][.95\textwidth]H_{n+1}(\beta^2_{n+1}+\beta^2_n+\beta_n\beta_{n+1}+(\beta_n+\beta_{n+1})(b_1+b_2)+b_1b_2+\gamma_{n+2}+\gamma_{n+1}+\gamma_n\\-n(\beta_{n+1}+\beta_N+\beta_{n-1}+b_1+b_2)+\pi^{[-2]}_{n-2}),
\end{multlined}\\
\psi^{(2)}_n&=H_{n+2}(\beta_n+\beta_{n+1}+\beta_{n+2}+b_1+b_2-n).
\end{align*}

For the diagonal matrix \(\pi^{[-2]}\) the nontrivial  entries  are
\begin{align}\label{eq:pi2}\pi^{[-2]}_{n-2}=\frac{\eta A(\beta_{n-1},\beta_n,\beta_{n+1},,\beta_{n+2},\gamma_n,\gamma_{n+1},\gamma_{n+2})+B(\beta_{n-1},\beta_n,\beta_{n+1},,\beta_{n+2},\gamma_n,\gamma_{n+1},\gamma_{n+2})}{C(\beta_n,\beta_{n+1},\eta)}
\end{align}

and, the sub-leading coefficient of the corresponding hypergeometrical  discrete orthogonal polynomials is 
\begin{multline*}
    p^1_{n-2}=\frac{\eta A(\beta_{n-1},\beta_n,\beta_{n+1},,\beta_{n+2},\gamma_n,\gamma_{n+1},\gamma_{n+2})+B(\beta_{n-1},\beta_n,\beta_{n+1},,\beta_{n+2},\gamma_n,\gamma_{n+1},\gamma_{n+2})}{C(\beta_n,\beta_{n+1},\eta)}\\+\beta_{n-2}-(n-1)\beta_{n-1}-\frac{(n-3)(n-4)}{2},
\end{multline*}
and the \(\pi^{[-3]}\) matrix entries  are
\begin{multline}\label{eq:pi3}
\pi^{[-3]}_{n-3}=p^1_{n-2}[(\eta+1)(\gamma_n-\gamma_{n+1})+\beta_n+\beta_{n-1}+\beta_{n-2}+b_1+b_2]+D_n(\beta_{n-2},\beta_{n-1},\beta_n,\beta_{n+1},\gamma_{n-1},\gamma_n,\gamma_{n+1})\\+(\eta+1)E_n(\beta_{n-2},\beta_{n-1},\beta_n,\beta_{n+1},\gamma_n,\gamma_{n+1})+(\eta-1)F_n(\beta_{n-1},\beta_n,\beta_{n+1},\gamma_{n-1},\gamma_n,\gamma_{n+1},\gamma_{n+2})\\+\eta G_n(\beta_{n-1},\beta_n,\beta_{n+1},\gamma_n,\gamma_{n+1})
\end{multline}

\end{theorem}
\begin{proof}
Matrix elements of $\psi^{(-3)},\psi^{(-2)},\psi^{(-1)},\psi^{(0)}$ are obtained using $\Psi=\sigma(J)H\Pi^{\top}$. On the other hand, the entries of $\psi^{(3)},\psi^{(2)},\psi^{(1)}$ and the alternative expression of \(\psi^{(0)}\) are gotten  from $\Psi=\Pi^{-1}H\theta(J^{\top})$.
The matrix \(\pi^{[-2]}\) entries  follow from the equation from the first subdiagonal obtained from $[\Psi H^{-1},\Lambda^{\top}\gamma]=\vartheta_{\eta}(\Psi H^{-1})$. Applying \(\vartheta _{\eta}\) to $\psi^{(0)}H^{-1}$, where $\psi^{(0)}$ is the first expression for the main diagonal, and equating to \(\psi^{(-1)}H^{-1}\), after cleaning it is obtained \(\pi^{[-2]}\). From \(\pi^{[-2]}\) using \(p^1_{n-2}=\pi^{[-2]}_{n-2}-\frac{n(n-1)}{2}-(n-1)\beta_{n-1}+\beta_{n-2}\) it is obtained $p^1_{n-2}$.
Finally, the matrix $\pi^{[-3]}$ is obtained  from the main diagonal of $[\Psi H^{-1},J]=\Psi H^{-1}$.

\end{proof}

\begin{theorem}(Compatibility conditions)
The following relations are satisfied
\begin{align}\label{eq:comp1}
    \pi^{[2]}_{n-1}-\pi^{[2]}_{n-2}&=n(1-\beta_n+\beta_{n-1}) ,\\
\label{eq:comp2}
    \pi^{[3]}_{n-2}-\pi^{[3]}_{n-3}&=\pi^{[2]}_{n-2}(\beta_{n-2}-\beta_n+1)+n\gamma_{n-1}-(n-1)\gamma_n ,\\
\label{eq:comp3}
    \vartheta_{\eta}\pi^{[-2]}_{n-2}&=(n-1)\gamma_n-n\gamma_{n-1} \\
\label{eq:comp4}\vartheta_{\eta}\pi^{[-3]}_{n-3}&=(\gamma_n-\gamma_{n-2})\pi^{[-2]}_{n-2}+(n-1)(\beta_{n-2}-\beta_{n-1}-1)\gamma_n
 \end{align}
\end{theorem}
\begin{proof}
Equation \eqref{eq:comp1} is obtained from the second subdiagonal and  the second superdiagonal of the resulting matrix of  $[\Psi H^{-1},J]=\Psi H^{-1}$, it is also obtained from the second subdiagonal of $[\Psi H^{-1},\Lambda^{\top}\gamma]=\vartheta_{\eta}(\Psi H^{-1})$. Equation
\eqref{eq:comp2} is obtained from the first superdiagonal and from the first subdiagonal of $[\Psi H^{-1},J]=\Psi H^{-1}$, it is also obtained from the first subdiagonal of $[\Psi H^{-1},\Lambda^{\top}\gamma]=\vartheta_{\eta}(\Psi H^{-1})$ using the alternative expression of \(\psi^{(0)}\). Then, 
\eqref{eq:comp3} is a derived from the first superdiagonal of  $[\Psi H^{-1},\Lambda^{\top}\gamma]=\vartheta_{\eta}(\Psi H^{-1})$ and Equation 
\eqref{eq:comp4} is gotten from the main diagonal of $[\Psi H^{-1},\Lambda^{\top}\gamma]=\vartheta_{\eta}(\Psi H^{-1})$.
\end{proof}

\begin{rem}
\begin{enumerate}
	\item 	The first three  equations are equivalent to
	\begin{align}\label{eq:comp1_bis}\pi^{[-2]}_{n-2}-\pi^{[-2]}_{n-1}&=n(\beta_{n-1}-\beta_n-1),\\
\label{eq:comp2_bis} \pi^{[-3]}_{n-2}-\pi^{[-3]}_{n-3}&=\pi^{[-2]}_{n-2}(1+\beta_n-\beta_{n-2})-(n-1)\gamma_n+n\gamma_{n-1},\\
 \vartheta_{\eta}p^1_{n-1}&=-\gamma_{n-1}.\notag
\end{align}

\item Notice that compatibility conditions do not depend of hypergeometric parameters $a$'s or $b$'s.
\item  Equations \eqref{eq:pi2} and \eqref{eq:comp1_bis} combine and give a constraint among $(\beta_{n-1},\beta_n,\beta_{n+1},\beta_{n+2},\gamma_n,\gamma_{n+1},\gamma_{n+2})$, while
Equations \eqref{eq:pi3} and \eqref{eq:comp2_bis} leads to a ligature between $(\beta_{n-2},\beta_{n-1},\beta_n,\beta_{n+1},\beta_{n+2},\gamma_{n-1},\gamma_n,\gamma_{n+1},\gamma_{n+2})$. With these two equations at hand two Laguerre--Freud equations of the type \eqref{eq:LF} can be derived.

\end{enumerate}
\end{rem}

\section*{Conclusions and outlook}

 Adler and van Moerbeke have throughly used the Gauss--Borel factorization of the moment matrix, see \cite{adler_moerbeke_1,adler_moerbeke_2,adler_moerbeke_4}, In their studies of integrable systems and orthogonal polynomials. We have applied this ideas 
    in different contexts, CMV orthogonal polynomials, matrix orthogonal polynomials, multiple orthogonal polynomials  and multivariate orthogonal, see  \cite{am,afm,nuestro0,nuestro1,nuestro2,ariznabarreta_manas0,ariznabarreta_manas01,ariznabarreta_manas2,ariznabarreta_manas3,ariznabarreta_manas4,ariznabarreta_manas_toledano}.
For a general overview see \cite{intro}. 

Recently \cite{Manas_Fernandez-Irrisarri} we we applied that approach to  study  the consequences of the Pearson equation on the moment matrix and Jacobi matrices. For that description a new banded matrix is required, the Laguerre--Freud structure matrix that encodes the Laguerre--Freud relations for  the recurrence coefficients. We have also found that the contiguous relations fulfilled generalized hypergeometric functions determining the moments of the weight described for the squared norms of the orthogonal polynomials a discrete Toda hierarchy known as Nijhoff--Capel equation, see \cite{nijhoff}.  In \cite{Manas} we studied  the role of Christoffel and Geronimus transformations for the description of the mentioned contiguous relations, and the use of the Geronimus--Christoffel transformations  to characterize the shifts in the spectral independent variable of the orthogonal polynomials. 
In \cite{Irrisari-Manas} we  deepened in that program and searched further for the discrete semi-classical cases, finding Laguerre--Freud relations for the recursion coefficients for three types of  discrete orthogonal polynomials of generalized Charlier, generalized Meixner and  generalized Hahn of type I cases. In this paper we concluded this program by getting the Laguerre--Freud structure matrices and equations for three families of hypergeometrical discrete orthogonal polynomials.

A question of general character that remains open is the possibility of explicit determination (using the methods of \cite{Irrisari-Manas} and this paper) of  Laguerre--Freud equations for arbitrary hypergeometrical families ${}_p F_q$ of discrete orthogonal polynomials. In  \cite{Irrisari-Manas} and this paper we have shown that is possible for the cases $(p,q)=(0,1),(1,1),(2,1),(1,2),(2,2),(3,2)$. For the future, 
we will   also extend these techniques to multiple discrete orthogonal polynomials  \cite{Arvesu_Coussment_Coussment_VanAssche} and study its relations with the transformations presented in \cite{bfm}
and  quadrilateral lattices \cite{quadrilateral1,quadrilateral2}.
%

\end{document}